\documentclass[12pt]{article}

\usepackage[usenames]{color} 
\usepackage{amssymb} 
\usepackage{amsmath} 
\usepackage[utf8]{inputenc} 
\usepackage{comment}
\usepackage{hyperref}
\usepackage{bbm}
\usepackage{epsfig}
\usepackage{subfigure}
\usepackage{enumerate}

\includecomment{problem}

\includecomment{solution}

\def\ed { \stackrel{d}{=} }
\def\E { {\mathbb{E}}}
\def\P { {\mathbb{P}}}

\def\Z { {\mathbb{Z}}}

\def\XX{ {\mathcal{X} }}

\def\1{\mathbbm{1} }

\newcommand{\Xn}{ (X_1, X_2, \cdots)}
\newcommand{\Xnvar}{ (X_n, n \ge 1)}

\def\euler#1#2{\genfrac{\langle}{\rangle}{0pt}{}{#1}{#2}}
\def \endprf{\hfill {\vrule height6pt width6pt depth0pt}\medskip} 
\newenvironment{proof}{\noindent {\bf Proof} }{\endprf\par}

\newtheorem{theorem}{Theorem}[section]

\newtheorem{corollary}[theorem]{Corollary}

\numberwithin{equation}{section}

\newtheorem{innercustomgeneric}{\customgenericname}
\providecommand{\customgenericname}{}
\newcommand{\newcustomtheorem}[2]{%
  \newenvironment{#1}[1]
  {%
   \renewcommand\customgenericname{#2}%
   \renewcommand\theinnercustomgeneric{##1}%
   \innercustomgeneric
  }
  {\endinnercustomgeneric}
}

\newcustomtheorem{customthm}{Theorem}
\newcustomtheorem{customlemma}{Lemma}

\title{Stationary $1$-dependent Counting Processes:\\ from Runs to Bivariate Generating Functions}
\author{Jim Pitman, Zhiyi You}

\begin{document}
	
\date{May 17, 2021}

\maketitle

\section*{Abstract}

We give a formula for the bivariate generating function of a stationary $1$-dependent counting process in terms of its run probability generating function, with a probabilistic proof.
The formula reduces to the well known bivariate generating function of the Eulerian distribution in the case of descents of a sequence of indepependent and identically distributed 
random variables. The formula is compared with alternative expressions from the theory of determinantal point processes and the combinatorics of sequences.

\section*{Keywords}

One-dependent process, bivariate generating function, run probability generating function, Eulerian numbers, dependence structures, determinantal point process.

\section{Introduction} \label{sec:intro}

\subsection{Counting one-dependent processes}

A discrete time stochastic process $\Xn$ is said to be \textit{$1$-dependent} if 
\begin{equation}
    (X_1, \ldots, X_{m-1} ) \text{ is independent of } (X_{m + 1}, \ldots, X_{ m + n } )
\end{equation}
for all positive integers $m$ and $n$. In contrast to a Markov chain, this independence requires no knowledge of current position. 
This dependence structure has been widely investigated in probability theory \cite{borodin2010adding, holroyd2016finitely,holroyd2017one}, and as a tool in statistics \cite{kamps1989chebyshev} and queuing systems \cite{henderson2001regenerative, sigman1990one}. 
Examples of $1$-dependent processes are provided by \textit{$2$-block factors}  \cite{ibragimov1971independent} generated by a function of two successive terms in an independent sequence. 
But not all $1$-dependent processes can be constructed this way: Aaronson et al. \cite{aaronson1989algebraic} explicitly gave a two-parameter family of $1$-dependent processes which cannot be expresses as $2$-block factors. 
Other examples of this kind arise in the theory of random colorings of integers developed by Holroyd and Liggett \cite{holroyd2016finitely,holroyd2017one}.

Here, we restrict attention to $1$-dependent processes $\Xn$ which are also \textit{stationary}, i.e. for all positive integers $n$, \begin{equation}
    (X_1, X_2, \ldots, X_n ) \ed (X_2, X_3, \ldots, X_{ n + 1 } ).
\end{equation}
In the case of an independent and identically distributed (\textit{i.i.d.}) sequence, the distribution of the count $S_n(A) = \sum_{k = 1}^n 1(X_k \in A )$ is binomial with parameters $n$ and $\P(X_n \in A)$, for any measurable subset $A$ of the
state space of the sequence. We describe here a bivariate generating function which determines the distribution of this counting variable $S_n(A)$ for any stationary $1$-dependent process $\Xn$.

\subsection{Bivariate generating functions}

Following the work of de Moivre on the distribution of the number of spots in a number of die rolls, the encoding of a sequence by its \textit{generating function} was exploited by Euler \cite{euler1letter} and many subsequent authors for combinatorial enumeration
\cite{goulden2004combinatorial,flajolet2001analytic}.
To describe the distribution of an integer-valued random variable, Laplace \cite{laplace1820theorie} introduced the \textit{probability generating function} \cite{deakin1981development}.
For a sequence of non-negative-integer-valued random variables $S_n$, let $Q_{S_n}$ denote the probability generating function of $S_n$:
\begin{equation}
    \label{eqn:pgf}
    Q_{S_n}(z):= \E z^{S_n} = \sum_{k=0}^\infty \P( S_n = k ) z^k,
\end{equation}
and let $Q(z,v)$ be the \textit{bivariate generating function} of distributions of $S_n$
\begin{equation}
    \label{eqn:bgf}
    Q(z,v): = \sum_{n \ge 0 } \sum_{k \ge 0} \P( S_n = k) z^k  v^n   = \sum_{n \ge  0} Q_{S_n}(z) v^n,
\end{equation}
for all $z, v$ such that the summation converges, including $|z| \le 1$ and $|v| < 1$.
This bivariate generating function determines the distribution of $S_n$ for every $n$ by extraction of the coefficient of $z^k v^n$ from $Q(z, v)$:
\begin{equation}
    \P (S_n = k) = [z^k v^n] Q(z, v), \quad n, k = 0, 1, 2, \cdots
\end{equation}
In our set up for counting processes, $S_n:= X_1 + \cdots + X_n$, where
$\Xnvar$ is an indicator sequence, so each count $S_n$ takes values in $\{0, 1, 2, \cdots, n\}$, and the
series \eqref{eqn:bgf} is absolutely convergent for $|z v | < 1$.
See e.g. \cite[Chapter III]{flajolet2001analytic} for further background on bivariate generating functions.

\subsection{Run probability generating functions} \label{subsec:run_pgf}

For an indicator sequence $\Xnvar$, define its \textit{$0$-run probabilities}

\begin{equation}
    q_0 := 1 \text{ and } q_n := \P( S_n = 0 ) = \P( X_1 = X_2 = \cdots = 0 ), \quad n = 1, 2, \ldots
\end{equation}
and the associated \textit{$0$-run probability generating function}
\begin{align}
\label{eqn:0run_pgf}
    Q(v) &:= \sum_{n = 0}^\infty q_n v^n = \sum_{n = 0}^\infty \P( S_n = 0) v^n = Q(0,v).
\end{align}

The \textit{$1$-run probability} sequence can be similarly defined, treated as the 0-run probability sequence for the \textit{dual} indicator sequence $(\hat{X_n} := 1 - X_n, n \ge 1)$, with counts $\hat{S_n} = n - S_n$:

\begin{equation}
    p_0 := 1 \text{ and } p_n := \P( S_n = n ) = \P( X_1 = X_2 = \cdots = 1 ), \quad n = 1, 2, \ldots
\end{equation}

The associated \textit{$1$-run probability generating function} is then for $0 \le v < 1$
\begin{align}
\label{eqn:1run_pgf}
    P(v) &:= \sum_{n = 0}^\infty p_n v^n = \sum_{n = 0}^\infty \P( S_n = n ) v^n = \hat Q(0,v) = \lim_{z \to 0} Q(z^{-1}, zv),
\end{align}
where $\hat Q(z, v)$ is the bivariate generating function of $\hat S_n$, and the last equality is by dominated convergence as $z \to 0$, using the evaluation
for $|zv| < 1$:
\begin{align}
\hat Q(z, v) &= \sum_{n \ge 0 } \sum_{k=0}^n \P( S_n = n-k ) z^k  v^n \\
        &= \sum_{n \ge 0} \sum_{k=0}^n \P( S_n = k ) z^{-k}  (zv)^n = Q(z^{-1}, zv)  .
\end{align}

In our case of a stationary $1$-dependent sequence of indicator variables $\Xnvar$, it is known \cite[Chapter 7.4]{polishchuk2005quadratic}\cite[Theorem 1]{aaronson1989algebraic} that the 
distribution of $S_n = X_1 + \cdots + X_n$, is uniquely determined by its sequence of $1$-run probabilities, or just as well by its sequence of $0$-run probabilities, through a determinantal formula
for the probability function of the random vector $(X_1, \ldots, X_n)$.
Our main result, presented in Section \ref{sec:main_result}, gives a formula for the bivariate generating function $Q(z,v)$ of distributions of $S_n$ in this case, which is simpler than might be expected from this determinantal formula.
The rest of this paper is organized as follows.
\begin{itemize}
    \item Section \ref{sec:descents} 
shows how the Eulerian bivariate generating function is obtained from our result in the case of descents.
    \item Section \ref{sec:examples} displays the bivariate generating function of some other stationary $1$-dependent processes.
    \item Section \ref{sec:determinantal} verifies our result from the perspective of determinantal point processes.
    \item Section \ref{sec:goulden_jackson} makes connection with a combinatorial result in Goulden and Jackson \cite{goulden2004combinatorial} and provides Corollary \ref{thm:counting_cor_2bf} which is suitable for counting a particular pattern in $2$-block factors.
    \item Section \ref{sec:comparison} compares our formula for the bivariate generating function in the stationary $1$-dependent case to 
similar formulae for  exchangeable or renewal processes, which are either known or easily derived from known results.
\end{itemize}

\section{Main result} \label{sec:main_result}

\begin{theorem} \label{thm:bgf_stat_1dep}
For a stationary $1$-dependent indicator sequence $\Xnvar$, the bivariate generating function $Q(z,v)$ of distributions of its partial sums $S_n$  
is determined either by the $0$-run probability generating function $Q(v)$, or by the $1$-run probability generating function $P(v)$, via the
formulae
\begin{equation}
    \label{eqn:bgf_stat_1dep}
    Q(z,v) 
    = \frac{ Q ( (1 - z) v  ) }{ 1 - z v Q ( (1 - z ) v ) }
    = \frac{ P ( - (1 - z) v  ) }{ 1 - v P ( - (1 - z ) v ) }.
\end{equation}
\end{theorem}

The particular case $z=0$ of \eqref{eqn:bgf_stat_1dep} reduces to the following known result:
\begin{corollary} [Involution {\cite[Proposition 7.4]{borodin2010adding}}] \label{thm:involution_stat_1dep}
In the setting of the previous theorem, for any stationary $1$-dependent indicator sequence, the $0$-run generating function $Q(v)$ and the $1$-run generating function
$P(v)$ determine each other via the involution  of formal power series
\begin{equation}
    \label{eqn:0run_stat_1dep}
    Q(v)  = \frac{ P ( - v  ) }{ 1 - v P ( -  v ) }  ; \qquad P(v)  = \frac{ Q ( - v  ) }{ 1 - v Q ( -  v ) }   .
\end{equation}
\end{corollary}

\begin{proof} (of Theorem \ref{thm:bgf_stat_1dep} and Corollary \ref{thm:involution_stat_1dep})

We will first prove the left equality in \eqref{eqn:bgf_stat_1dep}, rearranged as
\begin{equation}
    \label{qopened}
    Q(z,v) = Q( (1-z) v ) + z v \, Q( (1-z) v ) Q( z,v ),
\end{equation}
by establishing the corresponding identity of coefficients of powers of $v$, that is, 
\begin{align}
\label{qnid}
Q_{S_n}(z) &= [v^n] Q( (1-z)v ) + z \sum_{k=0}^{n-1} [v^k] Q( (1-z)v ) [v^{n-1-k}] Q(z, v).
\end{align}
Recall that $q_j := [v^j] Q(v)$, whence 
\begin{equation}
    [v^j] Q( (1-z)v ) = (1-z)^j [v^j] Q(v) = (1-z)^j q_j, \quad j = 0, 1, \ldots.
\end{equation}
So \eqref{qnid} for each $n = 1, 2, \ldots$, with $j = k-1$,  reduces to
\begin{align}
\label{eqn:method_of_marks_decomposition}    Q_{S_n}(z) &= (1-z)^n q_n + \sum_{k=1}^n \left((1-z)^{k-1} z \right) q_{k-1} Q_{S_{n-k}} (z),
\end{align}
which has the following interpretation.
For $0 \le z \le 1$, let $(Y_n, n \ge 1)$ be a sequence of i.i.d. Bernoulli($z$) random variables, also independent of $\Xnvar$. 
Employing van Dantzig's method of marks \cite{vandantzig1949methode},
treat $Y_n$ as a \textit{mark} on $X_n$: say the $n$-th item $X_n$ is {\em $z$-marked} if $Y_n=1$, and {\em non-$z$-marked} if $Y_n = 0$. By construction, $Q_{S_n}(z)$ is the probability that every success among the first $n$ trials is $z$-marked. 
In particular, if $z = 0$, every success in non-$z$-marked. Then the only way every success in the first $n$ trials is $z$-marked is if there are no successes. Hence $Q_{S_n}(0) = q_n$ is the probability of no successes in the first $n$ trials.
The identity \eqref{eqn:method_of_marks_decomposition}, decomposes the event that every success in the first $n$ trials is $z$-marked according to the value of $T_z := \min \{ n: Y_n = 1\}$, the index of the first $z$-mark. So
\begin{itemize}
\item $T_z$ has the geometric$(z)$ distribution $\P(T_z = k) = (1-z)^{k-1} z$ for $k = 1,2, \ldots$;
    \item On the event of probability $(1-z)^n$, that the first $z$-mark occurs at $T_z >n$, no trial among the first $n$ is allowed to be success, with probability $q_n$;
    \item On the event of probability $(1-z)^{k-1} z$, that the first $z$-mark occurs at $T_z = k$ for some $1 \le k \le n$, no trial among the first $k-1$ is allowed to be success, with probability $q_{k-1}$, and all success after the $k$-th (excluding the $k$-th) are $z$-marked, with probability $Q_{S_{n-k}}(z)$, with independence before and after the $k$-th trial by the assumption that $\Xnvar$ is $1$-dependent.
\end{itemize}
This proves the left equality of \eqref{eqn:bgf_stat_1dep}. To prove the right, it is easiest to prove Corollary \ref{thm:involution_stat_1dep}. Recall \eqref{eqn:1run_pgf},

\begin{equation}
    P(v) = \lim_{z \to 0} Q(z^{-1}, zv) = \lim_{z \to 0} \frac{ Q ( (z-1) v  ) }{ 1 - v Q ( (z-1) v ) } = \frac { Q(-v) } { 1 - vQ(-v) },
\end{equation}
which yields the $P$ identity in \eqref{eqn:0run_stat_1dep}. To see the $Q$ identity, simply replace $v$ with $-v$ in the last equation.

Lastly, the right equation of \eqref{eqn:bgf_stat_1dep} is obtained from the left one and the involution

\begin{align}
    Q(z,v) 
    &= \frac{ P ( - (1 - z) v  ) / (1 - (1-z) v P (- (1 - z) v)) }{ 1 - z v P ( - (1 - z ) v ) / (1 - (1-z) v P (- (1 - z) v)) }\\
    &= \frac{ P ( - (1 - z) v  ) }{ 1 - v P ( - (1 - z ) v ) }.
\end{align}

\end{proof}

\section{Application to descents} \label{sec:descents}

In this section, we present the example of \textit{Eulerian numbers}. We were led to the general formula for the bivariate generating function of
counts of a $1$-dependent indicator sequence by the algebraically simple form of the bivariate generating function of Eulerian numbers, whose probabilistic meaning is not
immediately obvious, but nicely explained by the above proof of Theorem \ref{thm:bgf_stat_1dep}.

It is well known that a large class of stationary $1$-dependent indicator sequences (though not all, see \cite{aaronson1989algebraic, burton19931}) may be constructed from an independent and identically distributed \textit{background sequence} $(Y_1, Y_2, \ldots)$, as \textit{two-block factors} 
\begin{equation}
    \label{eqn:2bf}
    X_n := 1 ( ( Y_n,Y_{n+1} ) \in B ),
\end{equation}
for some product-measurable subset $B$ of the space of pairs of $Y$-values, say $[0, 1]^2$ for $Y_i \sim $ Uniform$(0, 1)$.

An important example is provided by the sequence of {\em descents} $X_n:= 1(Y_n > Y_{n+1})$ for real-valued $Y_i$.
In particular, for $Y_i \sim $ Uniform$(0, 1)$ (or any continuous distribution) and $S_n := D_{n+1}$ counting the number of descents $Y_i > Y_{i+1}$ with $1 \le i \le n$:
\begin{equation} \label{eqn:0run_descents}
    \P( S_n = 0) = \P(S_n = n) = \P( Y_1 > \cdots > Y_{n+1} ) = \frac{1 }{ (n+1)! }   .
\end{equation}

So the run generating functions $Q(v)$ and $P(v)$ in this case are easily evaluated as
\begin{equation}
\label{eqn:0run_pgf_descents}
    Q(v)  = P(v) =  \sum_{n \ge  0} \frac{ v^n } { (n+1)!}  =  \frac{ e^v - 1 } {v }.  
\end{equation}

\subsection{Eulerian numbers}

The \textit{Eulerian numbers} $\euler{n}{k}$ are commonly defined  by the numbers of permutations of $[n] := \{1, 2, \cdots, n\}$ with exactly $k$ descents, i.e. $k$ adjacent pairs with first larger than the second \cite{graham1989concrete}.
So the count $\hat S_{n-1}$ of descents in a uniform random permutation of $[n]$ has the {\em Eulerian distribution}
\begin{equation}
    \P(\hat S_{n-1} = k) = \frac {1} {(n)!} \euler{n}{k}.
\end{equation}

Observe that this uniform permutation can be done by taking the \textit{ranks} of the i.i.d. background sequence $(Y_1, Y_2, \cdots, Y_n)$. Here, we say the rank of $Y_i$ is $k$ if and only if $Y_i$ is the $k$-th smallest among $Y_1, Y_2, \cdots Y_n$. Then, for $Y_i \sim $ Uniform$(0, 1)$, the ranks are almost surely a uniform permutation of $[n]$. Therefore, $\hat S_n$ has the same distribution as $S_n$ in \eqref{eqn:0run_descents}. Now, applying Theorem \ref{thm:bgf_stat_1dep} to the descents $X_n:= 1(Y_n > Y_{n+1})$ implies the following bivariate generating function
\begin{equation}
    \label{eqn:bgf_descents}
    Q(z,v) = \sum_{n=0}^\infty \sum_{k=0}^{n} \P(S_n = k) z^k v^n = \frac{ e^{(1-z) v } - 1 }{ v ( 1 - z e^{(1-z) v } )} = \frac{ e^v - e^{zv} }{ v ( e^{zv} - z e^v )},
\end{equation}
which is the classical bivariate generating function 
of the Eulerian numbers 
\cite{scoville1974generalized,comtet1974advanced,graham1989concrete,macmahon2004combinatory,hwang2020asymptotic}.

\section{More examples} \label{sec:examples}

To simplify displays,  we work in this section with the \textit{shifted run generating functions} 
\begin{equation} \label{eqn:shifted_run_pgf}
    \tilde Q(v) = 1 + v Q(v), \qquad \tilde P(v) = 1 + v P(v),
\end{equation}
and the \textit{shifted bivariate generating function} 
\begin{equation}
    \tilde Q(z,v) = 1 + v Q(z,v).
\end{equation}
For reasons which do not seem obvious, the algebraic form of the generating functions associated with a $1$-dependent
indicator sequence is typically simpler when they are shifted like this.
The shifted generating functions of some selected models are shown in the table below, with detailed explanation later.

There are some benefits for using the shifted generating functions.
Firstly, they are simpler, especially in the Eulerian case; secondly, for $2$-block factors, the shifted generating functions are actually `standard' in combinatorics, since $n$ is set to be the length of background sequence; 
thirdly, the formulae in Theorem \ref{thm:bgf_stat_1dep} and Corollary \ref{thm:involution_stat_1dep} are also slightly simplified:
the involution formula \eqref{eqn:0run_stat_1dep} becomes (see, e.g. \cite{macdonald1995symmetric,borodin2010adding} for earlier occurrences, and \cite{froberg1975determination, carlitz1976enumeration} where this formula was first discovered in the study of $2$-block factors)
\begin{equation} \label{eqn:0run_stat_1dep_var}
    \tilde Q(v) = \frac 1 {\tilde P(-v)}; \qquad \tilde P(v) = \frac 1 {\tilde Q(-v)},
\end{equation}
while our main theorem \eqref{eqn:bgf_stat_1dep} is re-written as
\begin{equation} \label{eqn:shifted_bgf_stat_1dep}
    \tilde Q(z, v) 
    = \frac {(1-z)\tilde Q((1-z)v)} {1 - z \tilde Q((1-z)v)} 
    = \frac {1-z} {\tilde P(-(1-z)v) - z}.
\end{equation}

\hspace{-0.5cm}
\begin{tabular}{ |c|p{4cm}|p{4cm}|p{5cm}|  }
 \hline
 \multicolumn{4}{|c|}{Table 1} \\
 \hline
 Model & Shifted $0$-run pgf $\tilde Q(v)$ & Shifted $1$-run pgf $\tilde P(v)$ & Shifted bgf $\tilde Q(z, v)$\\
 \hline
 Eulerian & $e^v$ & $e^v$ & $\frac{ 1-z }{ e^{-(1-z)v} -z }$\\
 I.i.d. & $\frac {1 + pv} {1 - (1-p)v}$ & $\frac {1 + v - pv} {1 - pv} $ & $\frac {1 + pv - pzv} {1 - v + pv - pzv}$ \\
 One-pair & $\frac {1 + pv} {1 - v + pv - pv^2 + p^2v^2}$ & $\frac {1 + v - pv - pv^2 + p^2v^2} {1 - pv}$ & $\frac {1 + pv -pvz} {1 - v + pv - pv^2 + p^2v^2 - pvz + pv^2z - p^2v^2z}$ \\
 Carries & $\left( 1 - \frac vb \right)^{-b}$ & $\left( 1 + \frac vb \right)^{b}$ & $\frac { 1-z } { \left( 1 - \frac {(1-z)v}b \right)^b - z }$\\
 Flipping & $\sqrt{\frac q p} \tan \left[ v \sqrt{ p q } \right.$ $\left.- \arctan \left( \frac q p \right) \right]$ &  $\left. \sqrt{\frac p q} \middle/ \tan \left[ -v \sqrt{ p q } \right. \right. $ $\left.- \arctan \left( \frac q p \right) \right]$ & $\frac { (1-z) \tan \left[ (1-z) v \sqrt{ p q } - \arctan \left( q / p \right) \right]  } { \sqrt {p / q} - z \tan \left[ (1-z) v \sqrt{ p q } - \arctan \left( q / p \right) \right]}$\\
 Non-2BF & $\frac 1 {1 - v + \alpha v^2 - \beta v^3}$ & $1 + v + \alpha v^2 + \beta v^3$ & $ \frac 1 {1 - v + \alpha (1-z) v^2 - \beta (1-z)^2 v ^3 }$\\
 \hline
\end{tabular}

We already discussed the \textit{Eulerian} model in Section \ref{sec:descents}. The rest of Table $1$ will be briefed here row by row. Usually, we only say how one of the run probability generating functions is obtained, since the other one and the bivariate generating function can then be found easily through \eqref{eqn:shifted_run_pgf}, \eqref{eqn:0run_stat_1dep_var} and \eqref{eqn:shifted_bgf_stat_1dep}.

\paragraph{Independent and identically distributed trials (I.i.d.)} \label{subsec:iid}

The classical example of i.i.d. Bernoulli($p$) trials is for sure an example of $1$-dependent sequence.

\paragraph{Indicator of two consecutive ones (One-pair)} \label{subsec:discrete_2bf}

Consider the simplest 2-block factors $X_n := 1(Y_n = Y_{n+1} = 1)$, where $(Y_n, n \ge 1)$ is i.i.d. Bernoulli($p$) trials. Its $1$-run probability generating function is easy to compute.

Considering the coefficient of $z^k v^n$ in its bivariate generating function gives the recursion
\begin{equation}
    q_{n, k} = (1-p) q_{n-1, k} + p q_{n-1, k-1} + p(1-p) (q_{n-2, k} - q_{n-2, k-1}), \qquad n \ge 2, k \ge 0, 
\end{equation}
where $q_{n, k} := [z^k v^n] Q(z, v) = [z^k v^{n+1}] \tilde Q(z, v) = \P(S_n = k)$ with initial values $q_{0, 0} = 1, q_{1, 0} = 1 - p^2, q_{1, 1} = p^2$ and convention $q_{n, k} = 0$ for $k > n$ or $k < 0$.

Setting $p = 1/2$ recovers the \textit{Fibonacci} sequence as its $0$-run probabilities:
\begin{equation}
    F_{n+2} =: 2^n q_{n, 0} = 2^{n-1} q_{n-1, 0} + 2^{n-2} q_{n-2, 0} = F_{n+1} + F_n,
\end{equation}
where $F_0 = F_1 = 1$ is the first two terms of the Fibonacci sequence we use here. This can also be interpreted as the chance of not getting any consecutive heads in a row of coin tosses, which has been recognized by many others, see \cite{mohanty1967coin, epstein1977theory, finkelstein1978fibonacci, honsberger1985second}.

\paragraph{Carries when adding a list of digits (Carries)} \label{subsec:carries}

Adding a list of digits using carries is discussed in \cite{borodin2010adding} as a stationary $1$-dependent process.
This example also falls into the category of $2$-block factors with i.i.d. Uniform($\{0, 1, \cdots, b-1\}$)'s as its background sequence, and $B = \{(x, y) : b > x > y \ge 0\}$.
Its $0$-run probabilities and generating function are explicitly given in \cite{borodin2010adding}.

\paragraph{Edge flipping on the integers (Flipping)} \label{subsec:edge_flipping}

Chung and Graham \cite{chung2012edge} introduced the following discrete time model of a random pattern in $\{0, 1\} ^V$ indexed by the vertex set $V$ of a finite simple graph $(V, E)$: pick an edge uniformly at random from $E$, then the pattern is updated by replacing its values on the two vertices joined with the picked edge by $11$ with probability $p$ and by $00$ otherwise, where the choices of edges and update of values on vertices are assumed to be all independent of the others, for some $0 < p < 1$.

They offered an analysis of discrete-time edge flipping on an $n$-cycle, which can be generalized in terms of a stationary continuous-time ($1$-dependent indicator) process indexed by the integers.

In short, we may sample its stationary distribution in the following manner:

\begin{itemize}
    \item first, generate a sequence $(U_n, n \in Z)$ of i.i.d. Uniform $(0, 1)$'s, which works as the time of last update on the edge $\{ n, n+1\}$. That is to say, if $U_n > U_{n-1}$, then the last update on edge $\{n, n+1\}$ happened later than the one on $\{ n-1, n \}$;
    \item secondly, generate a sequence $(W_n, n \in \Z)$ of i.i.d. Bernoulli $(p)$'s, independent of $(U_n, n \in Z)$, which stands for whether the last update on the edge $\{ n, n+1\}$ is $11$ or $00$.\
    \item lastly,
    \begin{equation}
        X_n := 1(U_n > U_{n-1}) W_n + 1(U_n < U_{n-1}) W_{n-1}, \quad n \in \Z.
    \end{equation}
\end{itemize}

This is apparently a stationary $2$-block factor. Its shifted $0$-run probability generating function is given in \cite[Theorem 6]{chung2012edge}.

\paragraph{A non-$2$-block-factor example (Non-2BF)}

Aaronson, Gilat, Keane and de Valk \cite{aaronson1989algebraic} first discovered this family of non-$2$-block-factor stationary $1$-dependent indicator processes. In short, they forbid the appearance of three consecutive ones, hence the $1$-run probability generating functions are as simple as quadratic functions. Note that not all value pairs $(\alpha, \beta)$ make this process not a $2$-block factor -- only some work, while some others do not yield stationary $1$-dependent processes at all. See \cite[Fig. 2]{aaronson1989algebraic}.

\section{Determinantal representation} \label{sec:determinantal}

Any indicator process can be treated 
as a point process by regarding the indicated events as points. It was shown in \cite[Theorem 7.1]{borodin2010adding} that any $1$-dependent point process on a segment of $\Z$ is a \textit{determinantal} process, as discussed further in \cite{daley2003introduction, daley2008introduction, borodin2010adding}.

Given a finite set $\XX$, a point process on $\XX$ is a probability measure $\P$ on $2 ^ \XX$.  
Its {\em correlation function} is the function of subsets $A \subseteq \XX$ 
 defined by $\rho(A) := \P(S : S \supseteq A)$. A point process is said to be \textit{determinantal} with \textit{kernel} $K(x, y)$ if

\begin{equation}
    \rho(A) = \det ( K(x, y) )_{x, y \in A},
\end{equation}
where the right hand side is the determinant of the $|A| \times |A|$ matrix with $K(x, y)$ on its $(x, y)$-th entry for $x, y \in A$. Now, we may state the following theorem from \cite{borodin2010adding}.

\begin{theorem}[Theorem 7.1 \cite{borodin2010adding}] \label{thm:1dep_det_kernel}
    Every $1$-dependent point process on a segment of $\Z$ is determinantal with kernel
    \begin{equation}
        K(x, y) = \sum_{r = 1}^{y-x+1} (-1) ^{r-1} \sum_{x = l_0 < l_1 < \cdots < l_r = y+1} \prod_{k=1}^r \rho([l_{k-1}, l_k) \cap \Z), \quad x \le y,
    \end{equation}
    $K(x, y) = -1$ for $x = y + 1$, and $K(x, y) = 0$ for $x \ge y + 2$.
\end{theorem}

In the stationary case, 
\begin{equation}
    \rho([a, b) \cap \Z) = \P(S_{b-a} = b-a) = p_{b-a},
\end{equation}
where $S_n$ and $p_n$ inherit the setup in Section \ref{sec:intro}. It is easy to see that then the kernel is also stationary, i.e.

\begin{equation}
    K(x, y) := k(y - x) = K(x + c, y + c), \qquad \forall c \in \Z.
\end{equation}

To better describe the kernel, consider the \textit{kernel generating function}

\begin{equation}
    G_k(v) := \sum_{n\in\Z} k(n) v^n = -v^{-1} + p_1 + (p_2 - p_1^2) v + \cdots
\end{equation}

Borodin, Diaconis and Fulman\cite[Corollary 7.3]{borodin2010adding} give the following relationship between the kernel generating function $G_k$ and the $1$-run probability generating function $P$

\begin{equation}
    G_k(v) P(v) = -\frac 1v.
\end{equation}

One last interesting (but not hard) result \cite[Theorem 4.1]{borodin2010adding} is that we may write the probability of any string pattern as a determinant. Consider $\XX = [n] := \{1, 2, \cdots, n\}$, then the string with exactly $k$ zeros at $0 < w_1 < w_2 < \cdots < w_k \le n$, which corresponds to the subset $A := \{w_1, w_2, \cdots, w_k\}^c$, has probability

\begin{equation}\label{eqn:pattern_det}
    \P(A) = \det ( p_{w_{j+1} - w_i - 1} ) _{0 \le i, j \le k},
\end{equation}
where $p_n = 0, \forall n < -1, p_0 = p_{-1} = 1$ and $w_0 = 0, w_{k+1} = n + 1$.

As shown in \cite{borodin2010adding}, this formula \eqref{eqn:pattern_det} is obtained by application of inclusion-exclusion formula on the correlation function.
Hence, we may recover our bivariate generating function \eqref{eqn:bgf_stat_1dep} as follows:

\begin{corollary} \label{thm:mult_pgf_det}
    The multivariate probability generating function of the indicators $I_i$ of location $i, i \in \XX = [n]$ is
    \begin{equation} \label{eqn:mult_pgf_det}
        G_{n} ({\bf z}) := G_{I_1, I_2, \cdots, I_n}(z_1, z_2, \cdots, z_n) = \det ( g_{i, j} ) _{0 \le i, j \le n},
    \end{equation}
    where $g_{i, j} = p_{j-i}, \forall i-j \neq 1$ and $g_{j+1, j} = p_{-1}-z_j = 1-z_j$.
\end{corollary}

\begin{corollary} \label{thm:pgf_det}
    The ordinary probability generating function of the number $S_{n}$ of ones in $\XX = [n]$ is
    \begin{equation} \label{eqn:pgf_det}
        G_{S_n} (z) = \det ( \hat p_{j - i} ) _{0 \le i, j \le n},
    \end{equation}
    where $\hat p_k = p_k, \forall k \neq -1$ and $\hat p_{-1} = p_{-1}-z = 1-z$.
\end{corollary}

The proof to Corollary \ref{thm:mult_pgf_det} is quite straightforward: just multiply \eqref{eqn:pattern_det} by $\prod_{i \in A} z_i$ and then sum it up over all subsets $A \subseteq [n]$.
Corollary \ref{thm:pgf_det} is then obvious by considering the ordinary generating function of $S_n = \sum_{i \in [n]} I_i$, i.e. treat all $z_i$ in \eqref{eqn:mult_pgf_det} as $z$.
We are also aware of the following known determinantal generating function formula for any $1$-dependent process (not necessarily stationary)
\begin{equation}
    G_{S_n}(z) = \det(I+(z-1)K),
\end{equation}
where $I$ is the identity matrix and $K$ is the determinantal correlation kernel given by Theorem \ref{thm:1dep_det_kernel}. This is not easy to simplify even for the stationary case. But from hindsight, one may check that this is essentially equivalent to \eqref{eqn:pgf_det}.
Lastly, one may also show Theorem \ref{thm:bgf_stat_1dep} by the Laplace expansion of the last column of the determinant on the right of \eqref{eqn:pgf_det}.

\section{Enumeration of sequences} \label{sec:goulden_jackson}

We have derived our main Theorem \ref{thm:bgf_stat_1dep} 
from a probabilistic point of view, without assuming the sequences to be $2$-block factors. But its enumerative Corollary \ref{thm:counting_cor_2bf} on integer-valued $2$-block factors can be deduced from a combinatorial result in Goulden and Jackson \cite[Section 4]{goulden2004combinatorial}.

Recall \eqref{eqn:2bf}, we defined $2$-block factors on $V$ based on the background sequence $(Y_1, Y_2, \ldots)$ as indicators of adjacent pairs $(Y_i, Y_{i+1})$ falling in some subset $B \subseteq V^2$, which obviously is a $1$-dependent indicator process. For example, in section \ref{sec:descents}, we discussed the case where $V = [0, 1]$ and $B = \{ (x, y) : 0 \le y \le x \le 1\}$.

Say a string on $\Z_+$ is a \textit{$B$-string} if each adjacent pair belongs to $B$. Suppose further that the background sequence $(Y_1, Y_2, \ldots)$ has i.i.d. uniform distribution on $[m] = \{1, 2, \cdots, m\}$. Then, the $1$-run probability $p_k$ can be used to count the number $m_k$ of $B$-strings of length $k$, i.e.
\begin{equation}
    m_k = m^k p_{k-1}, \quad k \ge 1,
\end{equation}
and $m_0 = 1$ by convention. Define the \textit{$B$-string generating function} as
\begin{equation}
    G(v) = \sum_{k=0}^\infty m_k v^k = 1 + \sum_{k=1}^\infty p_{k-1} m^k v^k = 1 + mv P(mv),
\end{equation}
where $P(v)$ is the $1$-run probability generating function.
Now the following corollary follows from Theorem \ref{thm:bgf_stat_1dep} by elementary algebra.

\begin{corollary} \label{thm:counting_cor_2bf}
    For $V = [m]$, the number of sequences of length $n \ge 1$ on $V$ with exactly $k \in [0, n-1]$ occurrences of pairs of adjacent components in $B$ is
    \begin{equation} \label{eqn:counting_bgf_2bf}
        [z^k v^n] \frac {z - 1} {z - G ((z-1) v)}.
    \end{equation}
\end{corollary}

To see the connection between Corollary \ref{thm:counting_cor_2bf} and \cite{goulden2004combinatorial}, we need the following definition from \cite{goulden2004combinatorial}. Let the \textit{$B$-string of length $k$ enumerator} be

\begin{equation}
    \gamma_k({\bf v}) = \gamma_k(v_1, v_2, \cdots) := \sum_{|\sigma| = k} \prod_i v_i ^ {\tau_i (\sigma)},
\end{equation}
where the sum is over all $B$-string $\sigma$ of length $k$ and $\tau_i(\sigma)$ is the number of times that component $i \in \Z_+$ appears, with $v_i$ being the counting variable associated with $i$.

Using the $B$-string of length $k$ enumerator, the following result for enumerating sequences with a certain number of occurrences of $B$ is shown in \cite{goulden2004combinatorial}:

\begin{theorem}[{\cite[Corollary 4.2.12]{goulden2004combinatorial}}] \label{thm:gj_section4}
    The number of sequences of length $n = \sum_i \tau_i$ with exactly $\tau_i$ $i$'s and $k (\le n-1)$ occurrences of pairs of adjacent components in $B$ is
    \begin{equation} \label{eqn:goulden_jackson_pi_1_seq}
        [z^k \prod_i v_i^{\tau_i}] \left\{ 1 - \sum_{j = 1}^ \infty (z-1)^{j-1} \gamma_j({\bf v}) \right\}^{-1}.
    \end{equation}
\end{theorem}

By setting $v_1 = v_2 = \cdots = v$, this becomes equivalent to \eqref{eqn:counting_bgf_2bf} since $\gamma_j(v{\bf 1}) = m_j v^j$. A more generalized version of Theorem \ref{thm:gj_section4}, known as the Goulden-Jackson cluster theorem, allows enumeration of adjacent components of any length (not just pairs), see \cite{goulden1979inversion,gessel2020application}.

In the case of descents, treated Section \ref{sec:descents}, the above discussion  does not apply directly, as since the background sequence is continuously distributed. However, this can be circumvented by considering the ranks in the partial sequence 
instead of the absolute values, hence it can be treated combinatorially as in \cite[Section 2.4.21]{goulden2004combinatorial} by considering only permutations of $\{1, 2, \cdots, n\}$ instead of all possible sequences.

An application of Corollary \ref{thm:counting_cor_2bf} can be found in \cite{florez2018further}.

\begin{theorem}[{$a = 4$: Florez \cite[Theorem 9]{florez2018further}}]
    The number of sequences on $V = \{0, 1, \cdots, a-1\}$ of length $n+m$ with exactly $m$ occurrences of adjacent pair $01$ is
    \begin{equation} \label{eqn:florez_example_generalized}
        f(n, m) = [x^m y^n]F(x, y) = [x^m y^n] \frac 1 {1 - (a+x)y + y^2}. 
    \end{equation}
\end{theorem}

To check this from Corollary \ref{thm:counting_cor_2bf}, note first that the $1$-run probability generating function is $P(v) = 1 + \frac 1 {a^2} v$, hence the $B$-string generating function is $G(v) = 1 + av + v^2$.
\eqref{eqn:florez_example_generalized} is then proved by substitutions $k \to m, n-k \to n, v \to y, zv \to x$.

\section{Comparison with other dependence structures} \label{sec:comparison}

\hspace{-0cm}
\begin{tabular}{ |c|c|c| }
 \hline
 \multicolumn{3}{|c|}{Table 2} \\
 \hline
 Dependence structure & $Q(z, v)$ in terms of $Q(v)$ & $Q(z, v)$ in terms of $P(v)$ \\
 \hline
 Stationary $1$-dependent & $\frac{ Q ( (1 - z) v  ) }{ 1 - z v Q ( (1 - z ) v ) }$ & $\frac{ P ( - (1 - z) v  ) }{ 1 - v P ( - (1 - z ) v ) }$ \\
 Exchangeable & $\frac { Q \left( \frac {(1-z)v} {1 - zv} \right) } {1 - zv}$ & $\frac { P \left( - \frac {(1-z)v} {1 - v} \right) } {1 - v}$ \\
 Renewal & $\frac {Q(v)} {1 - z(1 + (v-1)Q(v))}$ & N/A \\
 Stationary renewal & $\frac {- z + (z - v + (1-z)v Q'(0)) Q(v) } { z(v-1) + (1-z)v (Q'(0)-1) + z(v-1)^2 Q(v) }$ & N/A \\
 \hline
\end{tabular}

Apart from stationary $1$-dependent processes we have discussed in this article, there are several other dependence structures whose bivariate generating functions can be written in terms of run probability generating functions. It is interesting to 
compare these formulas, but no current theory seems to unify them. Some examples may even belong to more than one dependence structure, see Table 2 above and explanations below.

\paragraph{Exchangeable}

A sequence $\Xnvar$ is \textit{exchangeable} iff for each $n \ge 1$ and each permutations $\pi$ of $[n]$
\begin{equation}
    (X_{\pi(1)}, X_{\pi(2)}, \cdots, X_{\pi(n)}) \ed (X_1, X_2, \cdots, X_n).          
\end{equation}
Properties of exchangeable sequences were first studied by de Finetti \cite{de2017theory}. The i.i.d. example in Section \ref{subsec:iid} is also exchangeable.

\paragraph{Renewal and stationary renewal}

Consider the partial sum $S_n = X_1 + X_2 + \cdots + X_n$ of the sequence $\Xnvar$, and its inter-arrival times $T_1 = \min \{n : S_n = 1\}, T_k = \min \{n : S_n = k\} - T_{k-1}$. We say $\Xnvar$ is \textit{delayed renewal}, if the \textit{renewals} $T_2, T_3, \cdots$ are i.i.d., and the \textit{delay} $T_1$ is independent of $T_2, T_3, \cdots$. In particular, we say the indicator sequence is

\begin{enumerate}
    \item \textit{renewal conditional on $X_0 = 1$} or \textit{renewal}, if the delay $T_1$ has the same distribution as the renewals $T_2, T_3, \cdots$;
    \item \textit{stationary renewal}, if it is both stationary and delayed renewal.
\end{enumerate}

For these sequences it is not possible to recover the bivariate generating function from the $1$-run probability generating function $P(v)$, since in the delayed renewal case 
$$P(v) = 1 + q_1 ( v + r_1 v^2 + r_1^2 v^3 + \cdots )$$
 contains only the information about immediate renewals, which does not determine the distribution of $T_1$ or the $0$-run probabilities.

The bivariate generating formula is an indirect corollary of \cite[Theorem 3.5.4]{hunter1983mathematical}. See \cite{feller1957introduction} for introductions to renewal theory.

Three examples which appeared earlier in this article are both stationary $1$-dependent and stationary renewal: the indicator of two consecutive ones in Section \ref{subsec:discrete_2bf}, the case $b=2$ of carries when adding a list of digits in Section \ref{subsec:carries}, and the generalized Florez' example in Section \ref{sec:goulden_jackson}. All these three examples can be treated as indicators of time-homogeneous Markov chains visiting a particular state.

\bibliographystyle{amsplain}
\bibliography{Poly}

\providecommand{\bysame}{\leavevmode\hbox to3em{\hrulefill}\thinspace}
\providecommand{\MR}{\relax\ifhmode\unskip\space\fi MR }
\providecommand{\MRhref}[2]{%
  \href{http://www.ams.org/mathscinet-getitem?mr=#1}{#2}
}
\providecommand{\href}[2]{#2}
\begin{thebibliography}{10}

\bibitem{aaronson1989algebraic}
Jon Aaronson, David Gilat, Michael Keane, and Vincent de~Valk, \emph{An
  algebraic construction of a class of one-dependent processes}, Ann. Probab.
  \textbf{17} (1989), no.~1, 128--143. \MR{972778}

\bibitem{borodin2010adding}
Alexei Borodin, Persi Diaconis, and Jason Fulman, \emph{On adding a list of
  numbers (and other one-dependent determinantal processes)}, Bull. Amer. Math.
  Soc. (N.S.) \textbf{47} (2010), no.~4, 639--670. \MR{2721041}

\bibitem{burton19931}
Robert~M. Burton, Marc Goulet, and Ronald Meester, \emph{On {$1$}-dependent
  processes and {$k$}-block factors}, Ann. Probab. \textbf{21} (1993), no.~4,
  2157--2168. \MR{1245304}

\bibitem{scoville1974generalized}
L.~Carlitz and Richard Scoville, \emph{Generalized {E}ulerian numbers:
  combinatorial applications}, J. Reine Angew. Math. \textbf{265} (1974),
  110--137. \MR{429575}

\bibitem{carlitz1976enumeration}
L.~Carlitz, Richard Scoville, and Theresa Vaughan, \emph{Enumeration of pairs
  of sequences by rises, falls and levels}, Manuscripta Math. \textbf{19}
  (1976), no.~3, 211--243. \MR{432472}

\bibitem{chung2012edge}
Fan Chung and Ron Graham, \emph{Edge flipping in graphs}, Adv. in Appl. Math.
  \textbf{48} (2012), no.~1, 37--63. \MR{2845506}

\bibitem{comtet1974advanced}
Louis Comtet, \emph{Advanced combinatorics}, enlarged ed., D. Reidel Publishing
  Co., Dordrecht, 1974, The art of finite and infinite expansions. \MR{0460128}

\bibitem{daley2003introduction}
D.~J. Daley and D.~Vere-Jones, \emph{An introduction to the theory of point
  processes. {V}ol. {I}}, second ed., Probability and its Applications (New
  York), Springer-Verlag, New York, 2003, Elementary theory and methods.
  \MR{1950431}

\bibitem{daley2008introduction}
\bysame, \emph{An introduction to the theory of point processes. {V}ol. {II}},
  second ed., Probability and its Applications (New York), Springer, New York,
  2008, General theory and structure. \MR{2371524}

\bibitem{de2017theory}
Bruno de~Finetti, \emph{Theory of probability}, Wiley Series in Probability and
  Statistics, John Wiley \& Sons, Ltd., Chichester, 2017, A critical
  introductory treatment, Translated from the 1970 Italian original by Antonio
  Mach\'{\i} and Adrian Smith and with a preface by Smith, Combined edition of
  Vol. 1, 1974 and Vol. 2, 1975. \MR{3643261}

\bibitem{deakin1981development}
Michael A.~B. Deakin, \emph{The development of the {L}aplace transform,
  1737--1937. {I}. {E}uler to {S}pitzer, 1737--1880}, Arch. Hist. Exact Sci.
  \textbf{25} (1981), no.~4, 343--390. \MR{649489}

\bibitem{epstein1977theory}
Richard~A. Epstein, \emph{The theory of gambling and statistical logic},
  revised ed., Academic Press [Harcourt Brace Jovanovich, Publishers], New
  York-London, 1977. \MR{0446535}

\bibitem{euler1letter}
Leonhard Euler and Christian Goldbach, \emph{Leonhard {E}uler und {C}hristian
  {G}oldbach. {B}riefwechsel 1729--1764}, Abh. Deutsch. Akad. Wiss. Berlin Kl.
  Philos., Gesch., Staats-, Rechts-, Wirtschaftswiss \textbf{1965} (1965),
  no.~1, ix+420. \MR{0201260}

\bibitem{feller1957introduction}
William Feller, \emph{An introduction to probability theory and its
  applications. {V}ol. {I}}, John Wiley and Sons, Inc., New York; Chapman and
  Hall, Ltd., London, 1957, 2nd ed. \MR{0088081}

\bibitem{finkelstein1978fibonacci}
Mark Finkelstein and Robert Whitley, \emph{Fibonacci numbers in coin tossing
  sequences}, Fibonacci Quart. \textbf{16} (1978), no.~6, 539--541.

\bibitem{flajolet2001analytic}
Philippe Flajolet and Robert Sedgewick, \emph{Analytic combinatorics},
  Cambridge University Press, Cambridge, 2009. \MR{2483235}

\bibitem{florez2018further}
Rigoberto Fl\'{o}rez, Leandro Junes, and Jos\'{e}~L. Ram\'{\i}rez,
  \emph{Further results on paths in an {$n$}-dimensional cubic lattice}, J.
  Integer Seq. \textbf{21} (2018), no.~1, Art. 18.1.2, 27. \MR{3771670}

\bibitem{froberg1975determination}
Ralph Fr\"{o}berg, \emph{Determination of a class of {P}oincar\'{e} series},
  Math. Scand. \textbf{37} (1975), no.~1, 29--39. \MR{404254}

\bibitem{gessel2020application}
Ira~M. Gessel, \emph{An application of the {G}oulden-{J}ackson cluster
  theorem}, 2020.

\bibitem{goulden1979inversion}
I.~P. Goulden and D.~M. Jackson, \emph{An inversion theorem for cluster
  decompositions of sequences with distinguished subsequences}, J. London Math.
  Soc. (2) \textbf{20} (1979), no.~3, 567--576. \MR{561149}

\bibitem{goulden2004combinatorial}
Ian~P Goulden and David~M Jackson, \emph{Combinatorial enumeration}, Courier
  Corporation, 2004.

\bibitem{graham1989concrete}
Ronald~L. Graham, Donald~E. Knuth, and Oren Patashnik, \emph{Concrete
  mathematics}, Addison-Wesley Publishing Company, Advanced Book Program,
  Reading, MA, 1989, A foundation for computer science. \MR{1001562}

\bibitem{henderson2001regenerative}
Shane~G Henderson and Peter~W Glynn, \emph{Regenerative steady-state simulation
  of discrete-event systems}, ACM Transactions on Modeling and Computer
  Simulation (TOMACS) \textbf{11} (2001), no.~4, 313--345.

\bibitem{holroyd2017one}
Alexander~E. Holroyd, \emph{One-dependent coloring by finitary factors},
  vol.~53, 2017, pp.~753--765. \MR{3634273}

\bibitem{holroyd2016finitely}
Alexander~E. Holroyd and Thomas~M. Liggett, \emph{Finitely dependent coloring},
  Forum Math. Pi \textbf{4} (2016), e9, 43. \MR{3570073}

\bibitem{honsberger1985second}
Ross Honsberger, \emph{A second look at the {F}ibonacci and {L}ucas numbers},
  Mathematical gems III. Washington, DC: Math Assoc Amer (1985).

\bibitem{hunter1983mathematical}
Jeffrey~J. Hunter, \emph{Mathematical techniques of applied probability. {V}ol.
  2}, Operations Research and Industrial Engineering, Academic Press, Inc.
  [Harcourt Brace Jovanovich, Publishers], New York, 1983, Discrete time
  models: techniques and applications. \MR{719019}

\bibitem{hwang2020asymptotic}
Hsien-Kuei Hwang, Hua-Huai Chern, and Guan-Huei Duh, \emph{An asymptotic
  distribution theory for {E}ulerian recurrences with applications}, Adv. in
  Appl. Math. \textbf{112} (2020), 101960, 125. \MR{4023911}

\bibitem{ibragimov1971independent}
I.~A. Ibragimov and Yu.~V. Linnik, \emph{Independent and stationary sequences
  of random variables}, Wolters-Noordhoff Publishing, Groningen, 1971, With a
  supplementary chapter by I. A. Ibragimov and V. V. Petrov, Translation from
  the Russian edited by J. F. C. Kingman. \MR{0322926}

\bibitem{kamps1989chebyshev}
U.~Kamps, \emph{Chebyshev polynomials and least squares estimation based on
  one-dependent random variables}, Linear Algebra Appl. \textbf{112} (1989),
  217--230. \MR{976339}

\bibitem{laplace1820theorie}
Pierre~Simon Laplace, \emph{Th{\'e}orie analytique des probabilit{\'e}s},
  Courcier, 1820.

\bibitem{macdonald1995symmetric}
I.~G. Macdonald, \emph{Symmetric functions and {H}all polynomials}, second ed.,
  Oxford Mathematical Monographs, The Clarendon Press, Oxford University Press,
  New York, 1995, With contributions by A. Zelevinsky, Oxford Science
  Publications. \MR{1354144}

\bibitem{macmahon2004combinatory}
Percy~A. MacMahon, \emph{Combinatory analysis. {V}ol. {I}, {II} (bound in one
  volume)}, Dover Phoenix Editions, Dover Publications, Inc., Mineola, NY,
  2004, Reprint of {{\i}t An introduction to combinatory analysis} (1920) and
  {{\i}t Combinatory analysis. Vol. I, II} (1915, 1916). \MR{2417935}

\bibitem{mohanty1967coin}
S.~G. Mohanty, \emph{An {$r$}-coin tossing game and the associated partition of
  generalized {F}ibonacci numbers}, Sankhy\={a} Ser. A \textbf{29} (1967),
  207--214. \MR{217830}

\bibitem{polishchuk2005quadratic}
Alexander Polishchuk and Leonid Positselski, \emph{Quadratic algebras},
  University Lecture Series, vol.~37, American Mathematical Society,
  Providence, RI, 2005. \MR{2177131}

\bibitem{sigman1990one}
Karl Sigman, \emph{One-dependent regenerative processes and queues in
  continuous time}, Math. Oper. Res. \textbf{15} (1990), no.~1, 175--189.
  \MR{1038240}

\bibitem{vandantzig1949methode}
D.~van Dantzig, \emph{Sur la m\'{e}thode des fonctions g\'{e}n\'{e}ratrices},
  Le {C}alcul des {P}robabilit\'{e}s et ses {A}pplications. {C}olloques
  {I}nternational du {C}entre {N}ational de la {R}echerche {S}cientifique, no.
  13, Centre National de la Recherche Scientifique, Paris, 1949, pp.~29--45.
  \MR{0032964}

\end{thebibliography}

\end{document}